\documentclass[11pt,twoside]{amsart}
\usepackage{amsmath, amsthm, amscd, amsfonts, amssymb, graphicx, color}
\usepackage[bookmarksnumbered, colorlinks, plainpages]{hyperref}

\newtheorem{theorem}{Theorem}[section]

\theoremstyle{definition}
\newtheorem{definition}[theorem]{Definition}
\theoremstyle{remark}
\newtheorem{remark}[theorem]{Remark}
\newtheorem{example}[theorem]{Example}

\numberwithin{equation}{section}
\def\pn{\par\noindent}

\begin{document}



\vspace{1.3 cm}

\title{Existence theorem of finite Krasner hyperfields}
\author{Surdive Atamewoue Tsafack, Ogadoa Amassayoga, Babatunde O. Onasanya and Yuming Feng}



\maketitle

\begin{abstract}
The concern of this paper is to show that there  always exist Krasner hyperfields of order $n$, where $n$ is an integer greater than or equal to $2$.
\end{abstract}

\vskip 0.2 true cm

{\scriptsize
\hskip -0.4 true cm {\bf Key Words:} Krasner hyperrings, Krasner hyperfields.
}

\pagestyle{myheadings}
\markboth{\rightline {\scriptsize  First Author and Second Author}}
         {\leftline{\scriptsize Short title of the paper}}

\bigskip
\section{Introduction}
In  classical algebra, notions of  group, ring and field are very useful and are studied for various reasons, most especially in the  understanding of our environment. Kasner, a French mathematician, introduced a tool for the study of approximation of valued fields.   He constructed a hyperoperation in which composition of two elements yield a set. He got a structure which was called canonical hypergroup  \cite{K83}.  Canonical hypergroups are to generalize the notion of group. Later, two of Krasner's students, Mittas and Stratigopoulos , studied hyperrings and hyperfields \cite{M69, M72, M73, S69}. These structures introduced by Krasner and his students have been extensively studied by many others who are interested in the characterization of  finite Krasner hyperstructures because of their increasingly growing applications in various directions \cite{A17, CL03, S20}.

There have been various approaches to the studies and constructions of finite Krasner hyperfields as can be seen in \cite{Massouros1, AEH20, IJAZ20}. However, it is clear that not all Krasner hyperfields can be constructed in these ways. Some researchers were interested in finding all hyperfields for a given number of elements \cite{AEH20, IJAZ20}. Also, by using some computer programs, some other researchers just conjecture the existence of finite Krasner hyperfields of a given order \cite{AEH20, IJAZ20}. Although, some algorithms claim to have the capacity to construct finite Krasner hyperfields, but they do not prove their existence. In this work, we prove that there always exist finite Krasner hyperfields of integer order $n\geq2$.


\section{Preliminaries}
In this section, all preliminary definitions and results which will be used throughout this work are discussed. More on this can be seen in\cite{CL03, DL07}.\\
\indent In particular, let $H$ be a non-empty set and $\mathcal{P}^\ast(H)$ be the set of all non-empty subsets of $H$. The map $$\oplus : H \times H \longrightarrow \mathcal{P}^\ast(H),$$ defined by $$(x, y)\mapsto x\oplus y \subseteq \mathcal{P}^\ast(H),$$ is called a {\it hyperoperation} and the couple $(H, \oplus)$ is called a {\it hypergroupoid}.
\begin{definition}
A hypergroupoid $(H, \oplus)$ is called a {\it semihypergroup} if,  for all $a, b, c$ in $H$, $$(a\oplus b)\oplus c = a\oplus (b\oplus c).$$
\end{definition}

\begin{definition}
A hypergroupoid $(H, \oplus)$ is called a {\it quasihypergroup} if for all $a\in H$, we have $$a\oplus H = H\oplus a = H.$$
\end{definition}
\begin{definition}
A hypergroupoid $(H, \oplus)$ which is both a semihypergroup and a quasihypergroup is called a {\it hypergroup}.
\end{definition}
\begin{definition}\label{Canonical hypergroup}
A {\it canonical hypergroup} $(R,\oplus)$ is an algebraic structure in which the following axioms hold:\begin{itemize}
\item [1.] For any $x, y, z\in R$, $x\oplus(y\oplus z)=(x\oplus y)\oplus z$,
\item [2.] For any $x, y\in R$, $x \oplus y = y \oplus x$,
\item [3.] There exists an additive identity $0\in R$ such that $0 \oplus x = \{x\}$ for every $x\in R$.
\item [4.] For every $x\in R$ there exists a unique element $x'$ (an opposite of $x$ with respect to hyperoperation ``$\oplus$") in $R$ such that $0\in x \oplus x'$,
\item [5.] For any $x, y, z\in R$, $z\in x \oplus y$ implies $y\in x'\oplus z$ and $x\in z\oplus  y'$.\end{itemize}
\end{definition}

\begin{remark}
 Note that, in the classical group $(R, +)$, the concept of opposite of $x\in R$ is the same as inverse.
\end{remark}

\begin{definition}
Let $(R, +, \cdot)$ be a ring. For any element $ x\in R $, $0$ is said to be bilaterally absorbing if $$x\cdot0 =0\cdot x = 0.$$
\end{definition}
A canonical hypergroup with a multiplicative operation which satisfies the following conditions is called a {\it Krasner hyperring}.
\begin{definition}
An algebraic  hyperstructure  $(R, \oplus, \cdot)$,  where ``$\cdot$" is usual multiplication on $R$,  is called a {\it Krasner hyperring} when the following axioms hold:\begin{itemize}
\item [1.] $(R, \oplus)$ is a canonical hypergroup with $0$ as additive identity,
\item [2.] $(R, \cdot)$ is a semigroup having $0$ as a bilaterally absorbing element,
\item [3.] The multiplication ``$\cdot$" is both left and right distributive over the hyperoperation ``$\oplus$".\end{itemize}
\end{definition}

A Krasner hyperring is called {\it commutative} (with unit element) if $(R, \cdot)$ is a commutative semigroup  (with unit element) and such is denoted $(R, \oplus, \cdot, 0, 1)$.
\begin{definition}
 Let $(R, \oplus, \cdot, 0, 1)$ be a commutative Krasner hyperring  with unit such that $(R\setminus\{0\}, \cdot, 1)$ is a group. Then,  $(R, \oplus, \cdot, 0, 1)$ is called a Krasner hyperfield.
\end{definition}
\begin{example}\cite{K83} Consider a field $(F, +, \cdot)$ and a subgroup $G$ of $(F\setminus\{0\},\cdot)$. Take $H=F/G= \{aG|~ a \in F\}$ with the hyperoperation and the multiplication given by:
$$ \left\{
     \begin{array}{ll}
       aG \oplus bG = { \{\bar{c}=cG |~ \bar{c} \in aG + bG\}} & \hbox{} \\
       aG \cdot bG = abG & \hbox{}
     \end{array}
   \right.$$
\indent Then $(H, \oplus, \cdot)$ is a Krasner hyperfield.
\end{example}

\begin{example}
Define a set $F=\{0,1,a,b,c\}$ a hyperoperation ``$\oplus$" and multiplication operation ``$\cdot$" as in the following tables:\\
\\
\begin{center}$\begin{array}{c|c|c|c|c|c}
   \oplus & 0 & 1 & a & b & c \\\hline
  0 & \{0\} & \{1\} & \{a\} & \{b\} & \{c\} \\\hline
   1 & \{1\} & \{1\} & \{1,a\} & \{0,1,a,b,c\} & \{1,c\} \\\hline
   a & \{a\} & \{1,a\} & \{a\} & \{a,b\} & \{0,1,a,b,c\} \\\hline
   b & \{b\} & \{0,1,a,b,c\} & \{a,b\} & \{b\} & \{b,c\} \\\hline
   c & \{c\} & \{1,c\} & \{0,1,a,b,c\} & \{b,c\} & \{c\}
 \end{array}
$\end{center}

\begin{center}
 $\begin{array}{c|c|c|c|c|c}
       \cdot & 0 & 1 & a & b & c \\\hline
       0 & 0 & 0 & 0 & 0 & 0 \\\hline
       1 & 0 & 1 & a & b & c \\\hline
       a & 0 & a & b & c & 1 \\\hline
       b & 0 & b & c & 1 & a \\\hline
       c & 0 & c & 1 & a & b
     \end{array}
$\end{center}
\vspace{1.0cm}
Then $(F,\oplus,\cdot)$ is a Krasner hyperfield with 5 elements which can be easily proved with the help of a computer program.
\end{example}
 The following example is one of the well-known construction of Krasner hyperfields from fields. It is due to the work of Massouros \cite{Massouros1}.
\begin{example}\label{ex2}
Let $(F, +, \cdot)$ be a field.  Define on $F$ the hyperoperation $$a \oplus b = \{a, b, a + b\},$$ for $a\neq b$, and $a, b \in F^*$,  where $F^*$ denotes the set of non zero elements of $F$. Then
$$a \oplus 0 = 0 \oplus a = \{a\},$$ for all $a \in F$ and $$a \oplus a'= F,$$ for all $a \in F$, where $a'$ is the opposite of the element $a$ with respect to the operation ``$+$" on $F$, (that is $a'+a=a+a'=0_F$).\\
\indent Therefore $(F, \oplus, \cdot)$ is a Krasner hyperfield.
\end{example}

\section{Existence theorem of finite Krasner hyperfields}
In this section, we prove that for any given integer $n\geq2$, there exist finite Krasner hyperfields of order $n$.

\begin{theorem}
There  exist  always finite  Krasner hyperfields of integer order $n\geq2$.
\end{theorem}

\begin{proof}
Let $(F,+_F,\cdot_F,0_F,1_F)$ and $(K,+_K,\cdot_K,0_K,1_K)$ be two finite fields. Their orders are  powers of prime.  Then we can only have the Krasner hyperfields $(F,\oplus_F,\cdot_F,0_F,1_F)$ and $(K,\oplus_K,\cdot_K,0_K,1_K)$ whose orders are only powers of prime corresponding to exactly the same elements in the underlying finite fields. But, with the product of $(F,+_F,\cdot_F,0_F,1_F)$ and $(K,+_K,\cdot_K,0_K,1_K)$, 
we can construct Krasner hyperfieds whose orders can be any of the integers $n\geq 2$, not necessarily powers of prime.  Define on the cartesian product $F\times K$ the hyperoperation
$$\oplus:F\times K\rightarrow \mathcal{P}^*(F\times K)$$ by $$(a,b)\oplus(c,d)=\displaystyle{\bigcup_{i\in a\oplus_Fc,j\in b\oplus_Kd}}(i,j),$$ with $(a,b), (c,d)\in F\times K$.

Here $(A,B)=\{(a,b)|~ a\in A$ and $b\in B\}$, where $A$ is a subset of $F$ and $B$ is a subset of $K$. Also, we have $$(A,B)\oplus(a,b)=\displaystyle{\bigcup_{i\in A, j\in B}}(i,j)\oplus(a,b).$$
The multiplication on $F\times K$ is defined by $$(a,b)\cdot(c,d)=(a\cdot_Fc, b\cdot_Kd),$$ where $(a,b), (c,d)\in F\times K$.
 If $A$ is a subset of $F$ and $B$ is a subset of $K$, then $(a,b)\cdot(A,B)=(a\cdot_FA,b\cdot_KB)$.

Now we prove that the algebraic hyperstructure $(F\times K, \oplus,\cdot, (0_F,0_K), (1_F,1_K))$ is a Krasner hyperfield.

We first prove that $(F\times K, \oplus, (0_F,0_K))$ is a canonical hypergroup satisfying the properties in Definition \ref{Canonical hypergroup}.\\
(i) \textbf{Additive identity element}\\
$(0_F,0_K)$ is the additive identity element of $F\times K$. Let $(a,b)$ be an element of $F\times K$.
\begin{eqnarray*}
               (a,b)\oplus(0_F,0_K) &=& \bigcup(a\oplus_F0_F,b\oplus_K0_K) \\
                &=& (\{a\},\{b\}) \\
                &=& \bigcup(0_F\oplus_Fa,0_K\oplus_Kb) \\
                &=& (0_F,0_K)\oplus(a,b).
             \end{eqnarray*}
(ii) \textbf{Commutativity}\\
Let $(a,b), (c,d)$ be two elements of $F\times K$.  \begin{eqnarray*}
               (a,b)\oplus(c,d) &=& (a\oplus_Fc,b\oplus_Kd) \\
                &=& (c\oplus_Fa,d\oplus_Kb) \\
                &=& (c,d)\oplus(a,b).
             \end{eqnarray*}
(iii) \textbf{Associativity}\\
Let $(a,b), (c,d), (e,f)$ be three elements of $F\times K$.
\begin{eqnarray*}
 (a,b)\oplus((c,d)\oplus(e,f))\quad &=& (a,b)\oplus(c\oplus_Fe,d\oplus_Kf) \\
   &=& (a\oplus_F(c\oplus_Fe),b\oplus_K(d\oplus_Kf)) \\
   &=& ((a\oplus_Fc)\oplus_Fe,(b\oplus_Kd)\oplus_Kf) \\
   &=& ((a,b)\oplus(c,d))\oplus(e,f).
\end{eqnarray*}
Therefore
$$(a,b)\oplus((c,d)\oplus(e,f))=((a,b)\oplus(c,d))\oplus(e,f).$$
(iv) \textbf{Opposite element}\\
Let $(a,b)$ be an element of $F\times K$. Assume that $a'$ is the opposite of $a$ in $F$ and $b'$ is the opposite of $b$ in $K$. Then,
$$(0_F,0_K)\in(a\oplus_Fa',b\oplus_Kb')=(a,b)\oplus(a',b').$$
Since $a'$ and $b'$ are unique, then $(a',b')$ is the unique opposite of $(a,b)$ in $F\times K$.\\
(v) \textbf{Reversibility}\\
Let $(a,b), (c,d), (e,f)$ be three elements of $F\times K$, such that $(a,b)\in (c,d)\oplus(e,f)$. Then, $(a,b)\in(c\oplus_Fe,d\oplus_Kf)$ implies $a\in c\oplus_Fe$ and $b\in d\oplus_Kf$. From $a\in c\oplus_Fe$ we know $e\in c'\oplus_Fa$ and $c\in a\oplus_Fe'$. Also, from $b\in d\oplus_Kf$ we get $f\in d'\oplus_Kb$ and $d\in b\oplus_Kf'$. Furthermore, $e\in c'\oplus_Fa$ and $f\in d'\oplus_Kb$ implies $$(e,f)\in(c'\oplus_Fa,d'\oplus_Kb)=(c',d')\oplus(a,b).$$
In the same way, $c\in a\oplus_Fe'$ and $d\in b\oplus_Kf'$ implies $$(c,d)\in(a\oplus_Fe',b\oplus_Kf')=(a,b)\oplus(e',f').$$ Hence, if $(a,b)\in (c,d)\oplus(e,f)$, we have $(e,f)\in(c',d')\oplus(a,b)$ and $(c,d)\in(a,b)\oplus(e',f')$. So the reversibility holds. Thus, $(F\times K, \oplus, (0_F,0_K))$ is a canonical hypergroup. \\ \\
Secondly, we prove that $(F\times K, \cdot, (0_F,0_K))$ is a semigroup having $(0_F,0_K)$ as a bilaterally absorbing element.\\
(vi) \textbf{Semigroup}\\
It is trivial to show that  $(F\times K, \cdot, (0_F,0_K))$ is a groupoid. We only show that associativity holds, in which case it is a semigroup.  Let $(a,b), (c,d), (e,f)$ be three elements of $F\times K$. Then
\begin{eqnarray*}
  (a,b)\cdot((c,d)\cdot(e,f)) &=& (a,b)\cdot(c\cdot_Fe,d\cdot_Kf) \\
   &=& (a\cdot_F(c\cdot_Fe),b\cdot_K(d\cdot_Kf)) \\
   &=& ((a\cdot_Fc)\cdot_Fe,(b\cdot_Kd)\cdot_Kf) \\
   &=& (a\cdot_Fc, b\cdot_Kd)\cdot(e,f) \\
   &=& ((a,b)\cdot(c,d))\cdot(e,f).
\end{eqnarray*}
(vii) \textbf{Bilaterally absorbing element}\\
$(0_F,0_K)$ is a bilaterally absorbing element. Let $(a,b)$ be an element of $F\times K$. Then
\begin{eqnarray*}
 (a,b)\cdot(0_F,0_K) &=& (a\cdot_F0_F,b\cdot_K0_K) \\
   &=& (0_F,0_K) \\
   &=& (0_F\cdot_Fa,0_K\cdot_Kb) \\
   &=& (0_F,0_K)\cdot(a,b).
\end{eqnarray*}
Thirdly, we prove that multiplication ``$\cdot$" is both left and right distributive over the hyperoperation ``$\oplus$".\\
Let $(a,b), (c,d), (e,f)$ be three elements of $F\times K$.
\begin{eqnarray*}
  (a,b)\cdot((c,d)\oplus(e,f)) &=& (a,b)\cdot(c\oplus_Fe,d\oplus_Kf) \\
   &=& (a\cdot_F(c\oplus_Fe),b\cdot_K(d\oplus_Kf)) \\
   &=& ((a\cdot_Fc)\oplus_F(a\cdot_Fe),(b\cdot_Kd)\oplus_K(b\cdot_Kf)) \\
   &=& ((a\cdot_Fc),(b\cdot_Kd))\oplus((a\cdot_Fe),(b\cdot_Kf)) \\
   &=& ((a,b)\cdot(c,d))\oplus((a,b)\cdot(e,f)).
\end{eqnarray*}
In the same way $((c,d)\oplus(e,f))\cdot(a,b)=((c,d)\cdot(a,b))\oplus((e,f)\cdot(a,b))$.\\
Lastly, we prove that $((F\times K)\setminus\{(0_F,0_K)\},\cdot,(1_F,1_K))$ is a group. This is true because if $(F\setminus\{0_F\},\cdot_F,1_F)$ and $(K\setminus\{0_K\},\cdot_K,1_K)$ are groups, so is their cartesian product since the product of groups is a group. This completes the prove that $(F\times K, \oplus,\cdot, (0_F,0_K), (1_F,1_K))$ is a Krasner hyperfield. Since the cardinality of a finite field is a prime power \cite{M1896}, we use the fundamental theorem of arithmetic \cite{HW79} to conclude that there always exists a finite Krasner hyperfields of order $n\geq2$. This ends the proof.
\end{proof}

\begin{remark}
From the work of Ameri, Eyvazi and Hoskova-Mayerova \cite{AEH20}, we construct the following table.
\vspace{0.2cm}
\begin{center}$\begin{array}{c|c}
   Number ~of~ elements & number~ of~ finite~ Krasner~ hyperfields~ up~ to~ isomorphism \\\hline
   2 & 2 \\\hline
   3 & 5 \\\hline
   4 & 7 \\\hline
   5 & 27 \\\hline
   6 & 16
 \end{array}
$\end{center}
\vspace{0.3cm}
This table allows us to affirm that there is no uniqueness for finite Krasner hyperfields as for finite fields.
\end{remark}

\section{conclusion}
This work proves that there   always exists  a finite Krasner hyperfields of integer order $n \geq 2$. Now we can tackle the problems of construction and classification of finite Krasner hyperfields without asking the question of existence. So, for a given integer $n\geq2$, how many Krasner hyperfields of order $n$ will we have? And are there fast algorithm for the construction of all of them?   We will try to solve these problems in next papers.

\section*{Competing interests}
The authors declare that they have no conflict of interests.


\vskip 0.4 true cm

\begin{center}{\textbf{Acknowledgments}}
\end{center}
This work is supported by Foundation of Chongqing Municipal Key Laboratory of Institutions of Higher Education ([2017]3), Foundation of Chongqing Development and Reform Commission (2017[1007]), and Foundation of Chongqing Three Gorges University. \\ \\
\vskip 0.4 true cm



\bigskip
\bigskip


{\footnotesize \pn{\bf Surdive Atamewoue Tsafack}\; \\ {School of Three Gorges Artificial Intelligence, Chongqing Three Gorges University, Chongqing 404100, P. R. China; Department of Mathematics, Higher Teacher Training College, University of Yaounde 1, Cameroon}\\
{\tt Email: surdive\@yahoo.fr}\\

{\footnotesize \pn{\bf Ogadoa Amassayoga}\; \\ {Department of Mathematics, Faculty of Science, University of Yaounde 1, Cameroon}\\
{\tt Email: cogadoaamassayoga\@yahoo.fr}\\

{\footnotesize \pn{\bf Babatunde Oluwaseun Onasanya}\; \\ {, Key Laboratory of Intelligent
 Information Processing and Control, Chongqing Three Gorges University, Chongqing 404100, P. R. China; Department of Mathematics, University of Ibadan,  Nigeria}\\
{\tt Email: babtu2001\@yahoo.com}\\

{\footnotesize \pn{\bf Yuming Feng}\; \\ {, School of Computer Science and Engineering, Chongqing Three Gorges University, Chongqing 404100, P. R. China; School of Electronic Information Engineering, Southwest University, Chongqing 400715, P. R. China}\\
{\tt Email: yumingfeng25928\@163.com}\\
\end{document}